\documentclass[12pt]{article}

\usepackage{graphicx,amsmath,amssymb,mathrsfs,amsfonts,amsthm,color,slashbox,amsbsy,ragged2e}
\usepackage{hyperref}

\usepackage{lineno}

\usepackage[figuresright]{rotating}
\usepackage{amsmath,amsthm,amssymb}
\usepackage[mathscr]{eucal}
\usepackage[latin1]{inputenc}
\usepackage{graphicx}
\newcommand{\bq}{\begin{equation}}
\newcommand{\eq}{\end{equation}}
\newcommand{\bqs}{\begin{equation*}}
\newcommand{\eqs}{\end{equation*}}
\newcommand{\bqa}{\begin{eqnarray}}
\newcommand{\eqa}{\end{eqnarray}}
\newcommand{\bqas}{\begin{eqnarray*}}
\newcommand{\eqas}{\end{eqnarray*}}
\newcommand{\bc}{\begin{cases}}
\newcommand{\ec}{\end{cases}}
\newcommand{\bt}{\begin{thm}}
\newcommand{\et}{\end{thm}}
\newtheorem{definition}{Definition}
\newtheorem{theorem}{Theorem}

\newtheorem{corollary}[theorem]{Corollary}
\newtheorem{lemma}[theorem]{Lemma}
\newtheorem{example}[theorem]{Example}
\theoremstyle{definition}

\numberwithin{equation}{section} \numberwithin{theorem}{section}

\title{A new notion of majorization with applications to the comparison of extreme order statistics}
\author{\small Esmaeil Bashkar$^1$, Hamzeh Torabi$^1$, % {Hossein Javanshiri}$^2$, 
{Ali Dolati}$^1$, {F$ \acute{e} $lix Belzunce$^2$}\\
\small{$^1$Department of Statistics, Yazd University, Yazd, Iran}\\
\small {$^2$Department of Statistics and Operations Research,  University of Murcia, Murcia, Spain}}

\begin{document}
\date{}
\maketitle

%\begin{frontmatter}



\begin{abstract}
In this paper, we use a new partial order, called the $f$-majorization order. The new order includes as special cases the  majorization, the reciprocal majorization and the $ p $-larger orders. We provide a comprehensive account of the mathematical properties of the $f$-majorization order and give applications of this order in the context of stochastic comparison for extreme order statistics of independent samples following the Fr$\grave{\rm e}$chet distribution and scale model.  We discuss stochastic comparisons of series
systems with independent heterogeneous exponentiated scale components in terms of the usual stochastic order and the hazard rate order. We also derive new result on the usual stochastic order for the largest order statistics of samples having exponentiated scale marginals and Archimedean copula structure.
\end{abstract}

%\begin{keyword}


%\end{keyword}
%\end{frontmatter}

% \linenumbers

\section{Introduction}
 In the modern life, applications of order statistics can be found in numerous fields, for example in statistical inference,  life testing and reliability theory.  The first important work devoted to the stochastic comparisons of order statistics arising from heterogeneous exponential random variables is the one by Pledger and Proschan \cite{p}. Some other papers in this direction, and in particular devoted to the comparison of extreme order statistics from heterogeneous exponential distributions are \cite{dyketal97}, \cite{ps}, \cite{kk}. There are many other papers on the comparison of extreme order statistics for some other models of parametric distributions. For example \cite{k}, \cite{to15}, \cite{ll}  and \cite{tk15} deal with the case of heterogeneous Weibull distributions, \cite{fz15} and \cite{kc16} deal with with the case of heterogeneous exponentiated Weibull distributions, \cite{bee} deals with the case of heterogeneous  generalized exponential distributions and \cite{gup} deals with the  case of heterogeneous Fr$\grave{\rm e}$chet distributions. A recent review on the topic can be also found in \cite{bz}. 
  
In these applications, various notions of majorization are used very often.  The majorization orders which are used for finding some nice and applicable inequalities is also useful in understanding the insight of the theory. This concept deals with the diversity of the components of a vector in $ \mathbb R^{n} $. Another interesting weaker order related to the majorization orders introduced in \cite{bp} is the $ p $-larger order. In \cite{zb} the reciprocal majorization is introduced. Note that, for basic notation and terminologies on majorization where we use in this paper, we shall follow \cite{met}. Fang and Zhang \cite{ff} and Fang \cite{fff} used a notion of majorization to prove Slepian's inequality.  In this paper, we used this notion, ( which called $f$-majorization order ) and give applications of this order in the context of stochastic comparison of parallel/series systems with independent and dependent components. This notion includes as particular cases some of the previous ones.

The paper is organized as follows. In Section 2 we provide several notions of stochastic orders and majorization orders, and some known results. We also provide new notions of majorization, the relationships with the previous notions and some new lemmas that will be used later. In Section 3 we provide new results for the comparison of extreme order statistics from heterogeneous  Fr$ \grave{\rm e} $chet, scale  and the exponentiated scale
populations. We also derive new result on the usual stochastic order for the largest order statistics of the random samples having exponentiated scale marginals and Archimedean copula structure. To finish some conclusions are provided in Section 4.

Throughout this paper,  we use the notations $ \mathbb R = (-\infty,+\infty) $, $ \mathbb R_{+} = [0,+\infty) $ and $ \mathbb R_{++} = (0,+\infty) $ and
 the term increasing means nondecreasing and decreasing
means nonincreasing.  Also the notation $ X_{1:n} $ ($ X_{n:n} $) is used to denote the smallest (largest) order statistic of $ n $ random variables $ X_{1}, \ldots , X_{n} $. For any differentiable function $ f(\cdot) $, we
write $ f^{\prime}(\cdot) $ to denote the first derivative. The random variables
considered in this paper are all nonnegative.

\section{Preliminaries on majorization and new definitions}
In this section, we recall some notions of stochastic orders, majorization and related orders and some useful lemmas, which are helpful for proving our main results.

Let $  X $ and $ Y $ be univariate random variables with distribution functions $ F $ and $ G $, density functions $ f $ and $ g $, survival functions $ \overline{F}=1 - F  $ and $ \overline{G}=1 - G  $, hazard rate functions $ r_{F}=f / \overline{F} $ and $ r_{G}=g / \overline{G} $, and reversed hazard rate functions $ \tilde{r}_{F}= f / F $ and $ \tilde{r}_{g}=g /G$, respectively. Based on these functions several notions of stocahstic orders have been defined, to compare the magnitudes of two random variables. Next we recall some defintions of stocahstic orders that will be used along the paper. For a comprehensive discussion on various stochastic orders, please refer to \cite{ms}, \cite{ss}, \cite{lll} and more recently \cite{bmrm}.

\begin{definition}
{\rm Let $  X $ and $ Y $ be two random variables with common support $ \mathbb R_{++} $. $ X $ is said to be smaller than $ Y $ in the 
\begin{itemize}

\item[{\rm(i) }] dispersive order, denoted by $ X\leq_{\rm disp}Y $, if $  F^{-1}(\beta)-F^{-1}(\alpha)\leq G^{-1}(\beta)-G^{-1}(\alpha)$ for all $ 0<\alpha\leq\beta<1 $,
%\item[{\rm(ii)}] likelihood ratio order, denoted by $ X\leq_{lr}Y $, if $ g(x)/f (x) $ is increasing in $ x\in \mathbb R_{++} $,
\item[{\rm(ii)}] hazard rate order, denoted by $ X\leq_{hr}Y $, if $ r_{F} (x) \geq r_{G}(x) $ for all $ x $,
\item[{\rm(iii)}] reverse hazard rate order, denoted by $ X\leq_{rh}Y $, if $ \tilde{r}_{F} (x) \leq \tilde{r}_{G}(x) $ for all $ x $,
\item[\rm (iv)] usual stochastic order, denoted by $ X \leq_{\rm st} Y $, if $ \overline{F}(x) \leq \overline{G}(x) $ for all $ x $.
\end{itemize}}
\end{definition}

A real function $ \phi $ is $n$-monotone on $ (a, b) \subseteq (-\infty,+\infty) $ if $ (-1)^{n-2}\phi^{(n-2)} $ is
decreasing and convex in $ (a, b) $ and $ (-1)^{k}\phi^{(k)}(x)\geq0 $ for all $ x \in (a, b), k =
0, 1, \ldots , n - 2 $, in which $\phi^{(i)}(.)$ is the $i$th derivative of $\phi(.)$.
For a $n$-monotone $ (n \geq 2) $ function $ \phi : [0,+\infty) \longrightarrow [0, 1] $ with $ \phi(0) = 1 $ and $ \lim_{x\rightarrow+\infty} \phi(x) = 0 $, let $ \psi = \phi, ^{-1} $ be the
pseudo-inverse, then
\begin{equation*}
C_{\phi}(u_1, \ldots , u_n) = \phi (\psi(u_1) + \ldots + \psi(u_n)), \quad \text{for all}\, u_{i} \in [0, 1], i = 1, \ldots , n,
\end{equation*}
is called an Archimedean copula with the generator $ \phi $. Archimedean copulas
cover a wide range of dependence structures including the independence copula with the
generator $ \phi(t)=e^{-t} $. For more on Archimedean copulas, readers may refer to \cite{nels}
and \cite{ncn}.

Next we provide formal definitions of the different majorization notions that can be found in the literature. Note that, for basic notation and terminologies on majorization used  in this paper, we shall follow \cite{met}. To provide these notions, let us recall that the notation $ x_{(1)}\leq x_{(2)}\leq ...\leq x_{(n)}$ ($ x_{[1]}\geq x_{[2]}\geq ...\geq x_{[n]} $) is used to denote the increasing (decreasing) arrangement of the components of the vector $ \boldsymbol{x} = (x_{1}, \ldots , x_{n})$.

\begin{definition}\label{a1}
{\rm The vector $ \boldsymbol{x} $ is said to be
\begin{itemize}
\item[ (i)] weakly submajorized by the vector $ \boldsymbol{y} $  (denoted by $ \boldsymbol{x}\preceq_{\rm w}\boldsymbol{y} $) if
$\sum_{i=j}^{n}x_{(i)}\leq \sum_{i=j}^{n}y_{(i)}$  for all  $j = 1, \ldots , n $,

\item[ (ii)] weakly supermajorized by the vector $ \boldsymbol{y} $ (denoted by $ \boldsymbol{x}\stackrel{\rm w}{\preceq}\boldsymbol{y} $) if $ \sum_{i=1}^{j}x_{(i)}\geq \sum_{i=1}^{j}y_{(i)} $ for  all $ j = 1, \ldots , n $,

\item[ (iii)]  majorized by the vector $ \boldsymbol{y} $ (denoted by $ \boldsymbol{x}\stackrel{\rm m}{\preceq}\boldsymbol{y} $) if $ \sum_{i=1}^{n}x_{i}= \sum_{i=1}^{n}y_{i}$ and  $\sum_{i=1}^{j}x_{(i)}\geq \sum_{i=1}^{j}y_{(i)}$ for all  $j = 1, \ldots , n-1 $.
\item[(iv)] A vector $  \boldsymbol{x}$ in $ \mathbb R^{n}_{+} $ is said to be weakly log-majorized
by another vector $ \boldsymbol{y} $ in $ \mathbb R^{n}_{+} $ (denoted by $\boldsymbol{x}\underset{\log}{\preceq _{\rm w}}\boldsymbol{y}$
) if
\begin{equation}\label{19}
\prod_{i=1}^{j}x_{[i]}\leq \prod_{i=1}^{j}y_{[i]},\quad j = 1, \ldots , n;
\end{equation}
$  \boldsymbol{x}$  is said to be log-majorized
by  $ \boldsymbol{y} $ (denoted by $\boldsymbol{x}\underset{\log}{\preceq}\boldsymbol{y}$
) if \eqref{19} holds with
equality for $ k = n $.

\item[(v)] A vector $  \boldsymbol{x}$ in $ \mathbb R^{n}_{+} $ is said to be $ p $-larger than another vector $ \boldsymbol{y} $ in $ \mathbb R^{n}_{+} $ (denoted by $ \boldsymbol{x}\stackrel{\rm p}{\succeq}\boldsymbol{y} $) if
\[
\prod_{i=1}^{j}x_{(i)}\leq \prod_{i=1}^{j}y_{(i)},\quad \mbox{ for } j = 1, \ldots , n.
\]

\item[(vi)] A vector $  \boldsymbol{x}$ in $ \mathbb R^{n}_{+} $ is said to reciprocal majorized by another vector $ \boldsymbol{y} $ in $ \mathbb R^{n}_{+} $ (denoted by $ \boldsymbol{x}\stackrel{\rm rm}{\succeq}\boldsymbol{y} $) if
\[
\sum_{i=1}^{j}\frac{1}{x_{(i)}}\geq \sum_{i=1}^{j}\frac{1}{y_{(i)}},\quad j = 1, \ldots , n.
\]
\end{itemize}}
\end{definition}
The ordering introduced in definition \ref{a1} (iv), called log-majorization, was  defined by Weyl \cite{weyl} and studied
by Ando and Hiai \cite{ando}, who delved deeply into applications in matrix
theory. Note that weak log-majorization implies weak submajorization. See
5.A.2.b. of \cite{met}.
Bon and P\v{a}lt\v{a}nea \cite{bp} and Zhao and Balakrishnan \cite{zb}  introduced the order of $ p $-larger and reciprocal majorization, Respectively.
Here it should be noted that, for two vectors  $\boldsymbol{x}$ and $\boldsymbol{y}$, we have $$\boldsymbol{x}\stackrel{\rm p}{\preceq}\boldsymbol{y}\Longleftrightarrow(\log (x_{1}), \ldots , \log (x_{n}))\stackrel{\rm w}{\preceq}(\log (y_{1}), \ldots , \log (y_{n})).$$
It is well-known that (cf. \cite{kkkk} and \cite{kmm})
\[
\boldsymbol{x}\stackrel{\rm rm}{\preceq}\boldsymbol{y}\Longleftarrow\boldsymbol{x}\stackrel{\rm p}{\preceq}\boldsymbol{y}\Longleftarrow\boldsymbol{x}\stackrel{\rm w}{\preceq}\boldsymbol{y}\Longleftarrow\boldsymbol{x}\stackrel{\rm m}{\preceq}\boldsymbol{y}\Longrightarrow \boldsymbol{x}\preceq_{\rm w}\boldsymbol{y},\quad \mbox{ for } \boldsymbol{x},\boldsymbol{y}\in  \mathbb R^{n}_{+}.
\]
The following lemma is needed for proving the main result.

\begin{lemma}[Balakrishnan et al. \cite{bee}]\label{beel}
{\rm Let the function $ h : (0, \infty)\times (0,1)\longrightarrow (0, \infty) $ be defined as
\[
h(\alpha, t)=\frac{\alpha(1-t)t^{\alpha-1}}{1-t^{\alpha}}.
\]
Then,
\begin{itemize}
\item[(i)]  for each $ 0<\alpha\leq1 $, $ h(\alpha, t) $ is decreasing with respect to $ t $; and
\item[(ii)]  for each $ \alpha\geq1 $, $ h(\alpha, t) $ is increasing with respect to $ t $.
\end{itemize}
}
\end{lemma}

We now introduce the main tool for this work. The idea is to provide a new majorization notion that includes as particular cases some of the previous ones. Also it will be used to provide some new results for the comparison of extreme values for the Fr$\grave{\rm e}$chet distribution, scale, and the exponentiated scale model.

\begin{definition}\label{myd}
{\rm Let $ f: \mathbb A \longrightarrow \mathbb R $ be a real valued function. The vector $ \boldsymbol{x} $ is said to be
\begin{itemize}
\item[\rm (i)] weakly $f$-submajorized by the vector $ \boldsymbol{y} $, denoted by $ \boldsymbol{x}\preceq _{\rm wf}\boldsymbol{y} $,   if
$f(\boldsymbol{x}) \preceq _{\rm w} f(\boldsymbol{y})$

\item[\rm (ii)] weakly $f$-supermajorized by the vector $ \boldsymbol{y} $,   denoted by 
$\boldsymbol{x} \stackrel{\rm wf}{\preceq}\boldsymbol{y} $,   if $f(\boldsymbol{x}) \stackrel{\rm w}{\preceq}f(\boldsymbol{y})$

\item[\rm (iii)]  $f$-majorized by the vector $ \boldsymbol{y} $, denoted by $ \boldsymbol{x}\stackrel{\rm fm}{\preceq}\boldsymbol{y} $,   if 
$f(\boldsymbol{x})\stackrel{\rm m}{\preceq}f(\boldsymbol{y})$,
\end{itemize}
where $ f(\boldsymbol{x}) = (f(x_{1}), \ldots , f(x_{n})) $.
}
\end{definition}

It is easy to see that most of the previous majorization notions are examples of the previous notions for some particular choices of the function $f$. In particular we have: 
\begin{align*}
\boldsymbol{x}\stackrel{\rm m}{\preceq}\boldsymbol{y} & \Longleftrightarrow\boldsymbol{x}\stackrel{\rm fm}{\preceq}\boldsymbol{y} \quad when\quad f(t)=t,\\
\boldsymbol{x}\stackrel{\rm p}{\preceq}\boldsymbol{y} & \Longleftrightarrow\boldsymbol{x}\stackrel{\rm wf}{\preceq}\boldsymbol{y} \quad when\quad f(t)=\log(t),
\\
\boldsymbol{x}\stackrel{\rm rm}{\preceq}\boldsymbol{y} & \Longleftrightarrow\boldsymbol{x}\preceq _{\rm wf}\boldsymbol{y} \quad when\quad f(t)=\dfrac{1}{t},\\
\boldsymbol{x}\underset{\log}{\preceq _{\rm w}}\boldsymbol{y} & \Longleftrightarrow\boldsymbol{x}\preceq _{\rm wf}\boldsymbol{y} \quad when\quad f(t)=\log(t).
\end{align*}

The following lemma show the relation between $f$-majorization notion and usual majorization for various functions.

\begin{lemma}\label{m2}
{\rm \begin{itemize}

\item[(i)] If an increasing function $ f $ is convex, then $ \boldsymbol{x}\mathop  \preceq \limits^{{\mathop{\rm wf}} }\boldsymbol{y} $ implies $  \boldsymbol{x} \mathop  \preceq \limits^{{\mathop{\rm w}} } \boldsymbol{y} $,
\item[(ii)] If an increasing function $ f $ concave, then $ \boldsymbol{x}\preceq_{wf}\boldsymbol{y} $ implies $  \boldsymbol{x} \preceq_{w}\boldsymbol{y} $.
\end{itemize}}
\end{lemma}
\begin{proof}
The proof of this lemma follows easily from Theorem 5.A.2 of \cite{met}.
\end{proof}
According to Lemma \ref{m2}, all of the results which obtain for weak majorization are also true for $ f $-majorization.

An interesting special case of Definition \ref{myd} by taking the exponential function can be achieved. More precisely, we have the following definition.

\begin{definition}
{\rm A vector $  \boldsymbol{x}$ in $ \mathbb R^{n}_{+} $ is said to be weakly $ \exp $-majorized
by another vector $ \boldsymbol{y} $ in $ \mathbb R^{n}_{+} $ (denoted by $\boldsymbol{x}\underset{\exp}{\preceq _{\rm w}}\boldsymbol{y}$
) if
\begin{equation}\label{expe}
\sum_{i=j}^{n}e^{x_{(i)}}\leq \sum_{i=j}^{n}e^{y_{(i)}},   \quad  \mbox{ for } j = 1,   \ldots ,   n;
\end{equation}
$  \boldsymbol{x}$  is said to be $ \exp $-majorized
by  $ \boldsymbol{y} $ (denoted by $\boldsymbol{x}\underset{\exp}{\preceq}\boldsymbol{y}$
) if \eqref{expe} holds with
equality for $ k = n $.
}
\end{definition}

Note that weak sub-majorization implies weak $ \exp $-majorization. See 5.A.2.g. of \cite{met}. In the following example we see that weak $ \exp $-majorization does not imply weak sub-majorization.

\begin{example}{\rm
Let $ (x_{1}, x_{2}) = (0.5, 0.9) $ and  $ (y_{1}, y_{2}) = (1.08, 0.3) $.
Obviously $ (x_{1}, x_{2})\not \succeq_{\rm w}(y_{1}, y_{2}) $ and $  (y_{1}, y_{2})\not \succeq_{\rm w}(x_{1}, x_{2}) $, even though we have $ (x_{1},x_{2})\underset{\exp}{\preceq _{\rm w}}(y_{1},y_{2})  $.}
\end{example}

Next we provide an example, which shows that $\boldsymbol{x}\stackrel{\rm fm}{\preceq}\boldsymbol{y} $ does not imply $ \boldsymbol{x}\stackrel{\rm m}{\preceq}\boldsymbol{y} $.
\begin{example}
{\rm Let $ \boldsymbol{x}=(\sqrt{2},   5) $ and $ \boldsymbol{y}=(2,  \sqrt{23}) $,   then  $ \boldsymbol{y}\stackrel{\rm fm}{\preceq}\boldsymbol{x} $ with $ f(t)= t^{2}$,   but it is clear that $ \boldsymbol{y}$ is not  majorized by $\boldsymbol{x} $.}
\end{example}

The following example shows that $\boldsymbol{x}\stackrel{\rm wf}{\preceq}\boldsymbol{y} $ does not imply $ \boldsymbol{x}\stackrel{\rm w}{\preceq}\boldsymbol{y} $, necessarily.
\begin{example}
{\rm Let $ \boldsymbol{x}=(2,  3) $ and $ \boldsymbol{y}=(1,   5) $,   and  $f$ be  any increasing function that assigns $ -5,  1.5,  -4,  1 $  to $ 1,  5,  2,  3 $,    respectively. We   observe that $ \boldsymbol{x}\stackrel{\rm wf}{\preceq}\boldsymbol{y} $,   but  $ \boldsymbol{x}\not \stackrel{\rm w}{\preceq}\boldsymbol{y} $.
}
\end{example}

Next we provide a set of technical results that will be used along the paper. 
First we introduce a lemma, which will be needed to prove our main results and is of interest in its own right.

\begin{lemma}\label{mkl}
{\rm The function $ \varphi: \mathbb R^{n}_{+}\longrightarrow \mathbb R$ satisfies

\begin{itemize}
\item[(i)] \begin{equation}\label{meq}
\boldsymbol{x}\preceq _{\rm wf}(\stackrel{\rm wf}{\preceq})\boldsymbol{y} \Longrightarrow \varphi(\boldsymbol{x}) \leq \varphi(\boldsymbol{y})
\end{equation}
if and only if, $ \varphi(f^{-1}(a_{1}),   \ldots,   f^{-1}(a_{n})) $ is Schur-convex in $ (a_{1},   \ldots,   a_{n}) $
and increasing (decreasing) in $ a_{i} $,   for $  i = 1,   \ldots ,   n $,  
\item[(ii)] \[
\boldsymbol{x}\stackrel{\rm fm}{\preceq}\boldsymbol{y} \Longrightarrow \varphi(\boldsymbol{x}) \leq (\geq)\varphi(\boldsymbol{y})
\]
if and only if, 
 $ \varphi(f^{-1}(a_{1}), \ldots, f^{-1}(a_{n})) $ is Schur-convex (Schur-concave) in $ (a_{1}, \ldots, a_{n}) $, where $ a_{i}=f(x_{i}) $, for $  i = 1, \ldots , n $.

\end{itemize}

where $ a_{i}=f(x_{i}) $,   for $  i = 1,   \ldots ,   n $ and $ f^{-1}(y) = \inf\{x | f(x)\geq y \} $.}
\end{lemma}
\begin{proof}
\begin{itemize}
\item[(i)] Using definition \ref{myd},   we see that \eqref{meq} is equivalent to
\[
\boldsymbol{a}\preceq_{\rm w}(\stackrel{\rm w}{\preceq})\boldsymbol{b} \Longrightarrow \varphi(f^{-1}(a_{1}),   \ldots,   f^{-1}(a_{n})) \leq \varphi(f^{-1}(b_{1}),   \ldots,   f^{-1}(b_{n})),  
\]
where $ a_{i}=f(x_{i}) $ and $ b_{i}=f(y_{i}) $,   for $  i = 1,   \ldots ,   n $. Taking 
\[
\phi(a_{1},   \ldots,   a_{n})=\varphi(f^{-1}(a_{1}),   \ldots,   f^{-1}(a_{n}))
\]
in Theorem 3.A.8 of \cite{met},   we get the required result.
\item[(ii)] This case can be proved in a very similar manner.
\end{itemize}
\end{proof}

It is noteworthy that Lemma 2.1 provided by Khaledi and Kochar \cite{kkkk} is a special case of Lemma \ref{mkl} (i) when $ f(x)=\log(x) $ and is useful for proving stochastic orders, see \cite{kh} and \cite{bee}. 
Recall that a real valued function $ \varphi $ defined on a set $ \mathscr{A}\in {\mathbb R}^{n} $ is said to be Schur-convex (Schur-concave) on $ \mathscr{A} $ if
\[
\boldsymbol{x} \stackrel{\rm m}{\preceq}\boldsymbol{y} \quad on\quad \mathscr{A} \Longrightarrow \varphi(\boldsymbol{x})\leq (\geq)\varphi(\boldsymbol{y}).
\]

\section{Applications to the comparison of extreme order statistics}

In this section we provide new results for the comparison of extreme values from independent Fr$\grave{\rm e}$chet distribution, scale and exponentiated scale model. We also derive new result on the usual stochastic order for largest order statistics of samples having exponentiated scale marginals and Archimedean copula structure. As we will see the main tools are the new $f$-majorization notions introduced in the previous section.  

\subsection{Comparison of extreme order statistics for the  Fr$\grave{\rm e}$chet distribution}

A random variable $X$ is said to be distributed according to the Fr$\grave{\rm e}$chet distribution, and will be denoted by $X\sim  {\rm Fr}\grave{\rm e}(\mu, \lambda, \alpha)  $, if the distribution function is given by 
\[
 G(x; \mu, \lambda, \alpha)=\exp\left\{-\left(\dfrac{x-\mu}{\lambda}\right)^{-\alpha}\right\},\quad x>\mu,
\]
where $\mu\in \mathbb R$ is a  location parameter, $ \lambda>0 $ is a scale parameter  and $ \alpha>0$ is a shape parameter. 

In this section, we discuss stochastic comparisons of series and parallel systems with Fr$\grave{\rm e}$chet distributed components in terms of the hazard rate order and the reverse hazard rate order. The result established here strengthens and generalizes some of the results of \cite{gup}. 
To begin with we present a generalization of Theorem 2 of  \cite{gup} where sufficient condition is based on the weak $f$-majorization.
This theorem provides the stochastic comparison result for the lifetime
of the parallel systems having independently distributed Fr$\grave{\rm e}$chet components
with varying scale parameters, but fixed location and shape parameters.

\begin{theorem}\label{f3}
{\rm Let $ X_{1}, \ldots , X_{n} $ ($ X^{*}_{1}, \ldots , X^{*}_{n} $) be independent random variables where $ X_{i} \sim {\rm Fr}\grave{\rm e}(\mu, \lambda_{i}, \alpha) $ ($ X^{*}_{i}\sim {\rm Fr}\grave{\rm e}(\mu, \lambda^{*}_{i}, \alpha) $), 
$ i = 1, \ldots , n $. Let us consider an strictly decreasing (increasing)  function $f$. 
\begin{itemize}
\item[(i)] If $ (f^{-1}(\cdot))^{\prime}(f^{-1}(\cdot))^{\alpha-1} $ is increasing (decreasing) and $(\lambda_1,\dots, \lambda_n) \stackrel{\rm wf}{\succeq}
(\lambda^{*}_{1}, \ldots , \lambda^{*}_{n}) $ then $ X_{n:n} \geq_{\rm rh} (\leq_{\rm rh}) X^{*}_{n:n} $.

\item[(ii)] If $ (f^{-1}(\cdot))^{\prime}(f^{-1}(\cdot))^{\alpha-1} $ is decreasing (increasing) and $ (\lambda^*_1,\dots, \lambda^*_n) \succeq _{\rm wf}
(\lambda_{1}, \ldots , \lambda_{n}) $ then $ X_{n:n} \geq_{\rm rh} (\leq_{\rm rh}) X^{*}_{n:n} $.
\end{itemize}
 }
\end{theorem}
\begin{proof}

\begin{itemize}

\item[(i)] Let us consider a fixed $x>0$, and a strictly monotone function $f$, then the reversed hazard rate of $ X_{n:n}$ is given by 
\[
\tilde{r}_{X_{n:n}}(x; \mu, \boldsymbol{a}, \alpha)=\sum_{i=1}^{n}\dfrac{\alpha}{f^{-1}(a_{i})}\left(\dfrac{x-\mu}{f^{-1}(a_{i})}\right)^{-\alpha-1},\quad x>\mu.
\]

From Lemma \ref{mkl}, teh proof follows if we prove that, for each $ x>0 $, $\tilde{r}_{X_{n:n}}(x; \mu, \boldsymbol{a}, \alpha) $  is Schur-convex (Schur-concave) and decreasing (increasing) in $ a_{i}^{\prime} $s.

Let $ h(a_{i})=\alpha(x-\mu)^{-\alpha-1}(f^{-1}(a_{i}))^{\alpha} $. By the assumption, $ f $ is a strictly decreasing (increasing) function, therefore we have
\[
\frac{\partial h(a_{i})}{\partial a_{i}}=\alpha^{2}(x-\mu)^{-\alpha-1}  \frac{\partial f^{-1}(a_{i})}{\partial a_{i}}(f^{-1}(a_{i}))^{\alpha-1}\leq (\geq) 0.
\]

Hence the reverse hazard rate function of $ X_{n:n} $ is decreasing (increasing) in each $ a_{i} $. 

Now, from Proposition 3.C.1 of \cite{met}, the Schur-convexity (Schur-concavity) of of $ \tilde{r}_{X_{n:n}}(x; \mu, \boldsymbol{a}, \alpha) $, follows if we prove the  convexity (concavity) of  $ h $.  The convexity (concavity) of $h$ follows from the assumption $ (f^{-1}(\cdot))^{\prime}(f^{-1}(\cdot))^{\alpha-1}  $ is increasing (decreasing). This completes the proof of the required result.

\item[(ii)] The proof is similar to the proof of part (i) and hence is omitted.
\end{itemize}
\end{proof}

Let us describe some particular cases of previous theorem. 

In Theorem \ref{f3}, if we let $ f(x)=\frac{1}{x} $, we can get the following corollary that generalizes
the corresponding result in Theorem 2 of \cite{gup}. In particular the majorization assumption is relaxed to the weak majorization and the usual stochastic order is replaced by the stronger reversed hazard rate order.

\begin{corollary}
 {\rm Let $ X_{1}, \ldots , X_{n} $ ($ X^{*}_{1}, \ldots , X^{*}_{n} $) be independent random variables where $ X_{i} \sim {\rm Fr}\grave{\rm e}(\mu, \lambda_{i}, \alpha) $ ($ X^{*}_{i}\sim {\rm Fr}\grave{\rm e}(\mu, \lambda^{*}_{i}, \alpha) $),
$ i = 1, \ldots , n $. If $ (\dfrac{1}{\lambda_1},\dots, \dfrac{1}{\lambda_n}) \stackrel{\rm w}{\succeq}
(\dfrac{1}{\lambda^{*}_{1}}, \ldots , \dfrac{1}{\lambda^{*}_{n}}) $ then $ X_{n:n} \geq_{\rm rh} X^{*}_{n:n} $.}
\end{corollary}

In Theorem \ref{f3}, if we let $ f(x) = x $, we can easily get the following result.

\begin{corollary}\label{f4}
{\rm
  Let $ X_{1}, \ldots , X_{n} $ ($ X^{*}_{1}, \ldots , X^{*}_{n} $) be independent random variables where $ X_{i} \sim {\rm Fr}\grave{\rm e}(\mu, \lambda_{i}, \alpha) $ ($ X^{*}_{i}\sim {\rm Fr}\grave{\rm e}(\mu, \lambda^{*}_{i}, \alpha) $),
$ i = 1, \ldots , n $. 
\begin{itemize}
\item[(i)] If $ \alpha\geq1 $, and $ (\lambda_1,\ldots,\lambda_n) \succeq_{\rm w}
(\lambda^{*}_{1}, \ldots , \lambda^{*}_{n}) $ then $ X_{n:n} \geq_{\rm rh} X^{*}_{n:n} $.

\item[(ii)] If $ 0<\alpha\leq1 $ and $ (\lambda_1,\dots,\lambda_n) \stackrel{\rm w}{\succeq}
(\lambda^{*}_{1}, \ldots , \lambda^{*}_{n}) $ then $ X_{n:n} \leq_{\rm rh} X^{*}_{n:n} $.
\end{itemize}
}
\end{corollary}

The following theorem present a generalization of Theorem 1 of  \cite{gup} where sufficient condition is based on the weak submajorization and by Lemma \ref{m2} (ii) is true under weak \textit{f} submajorization for any increasing concave function $f$ of the location parameters. 
 
\begin{theorem}\label{f1}
 \rm Let $ X_{1}, \ldots , X_{n} $ ($ X^{*}_{1}, \ldots , X^{*}_{n} $) be independent random variables where $ X_{i} \sim {\rm Fr}\grave{\rm e}(\mu_{i}, \lambda, \alpha) $ ($ X^{*}_{i}\sim {\rm Fr}\grave{\rm e}(\mu^{*}_{i}, \lambda, \alpha) $),
$ i = 1, \ldots , n $. If $ (\mu_1,\dots,\mu_n) \succeq_{\rm w}
(\mu^{*}_{1}, \ldots , \mu^{*}_{n}) $, then $ X_{n:n} \geq_{\rm rh} X^{*}_{n:n} $.
\end{theorem}

\begin{proof} It can be seen that the reversed hazard rate of $X_{n:n}$ is given by 
\[
\tilde{r}_{X_{n:n}}(x; \boldsymbol{\mu}, \lambda, \alpha)=\sum_{i=1}^{n}\dfrac{\alpha}{\lambda}\left(\dfrac{x-\mu_{i}}{\lambda}\right)^{-\alpha-1},\quad x>\max(\mu_{1}, \ldots , \mu_{n}).
\]

From Theorem 3.A.8 of \cite{met}, the proof follows if we prove that $ \tilde{r}_{X_{n:n}}(x; \boldsymbol{\mu}, \lambda, \alpha) $  is Schur-convex and increasing in $ \mu_{i}^{\prime} $s.

Let
\[
h(\mu_{i})=\dfrac{\alpha}{\lambda}\left(\dfrac{x-\mu_{i}}{\lambda}\right)^{-\alpha-1},
\]
then we have
\[
\frac{\partial h(\mu_{i})}{\partial \mu_{i}}=\dfrac{\alpha}{\lambda^{-\alpha}}(\alpha+1)(x-\mu_{i})^{-\alpha-2}  \geq0.
\]

Therefore the reverse hazard rate function of $ X_{n:n} $ is increasing in each $ \mu_{i} $. 

Now, from Proposition 3.C.1 of \cite{met}, we only need to prove the convexity of  $ h $ to get the Schur-convexity of $ \tilde{r}_{X_{n:n}}(x; \boldsymbol{\mu}, \lambda, \alpha) $. 

In this case, we have that
\begin{align*}
\frac{\partial^{2} h(\mu_{i})}{\partial \mu_{i}^{2}} & =\dfrac{\alpha}{\lambda^{-\alpha}}(\alpha+1)(\alpha+2)(x-\mu_{i})^{-\alpha-3}.
\end{align*}

Therefore we have that $ h $ is a convex function. This completes the proof.
\end{proof}

Note that $ (\mu_1,\dots,\mu_n)  \stackrel{\rm m}{\succeq}
(\mu^{*}_{1}, \ldots , \mu^{*}_{n}) $ implies $ (\mu_1,\dots,\mu_n) \succeq_{\rm w}
(\mu^{*}_{1}, \ldots , \mu^{*}_{n}) $,
Theorem \ref{f1} substantially improves Theorem 1 of  \cite{gup}.

\subsection{\rm Comparison of extreme values for scale model}

Independent random variables $ X_{1}, \ldots , X_{n} $ are said to belong to the scale family of distributions if  $ X_i\sim G(\lambda_i x)  $ where $ \lambda_i>0 $, $  i=1,\dots,n $ and $ G $ is called the baseline distribution and is an absolutely continuous distribution function with density function $ g $.
In the Theorem \ref{st1} we extend result of theorem 2.1 of \cite{kh} to the case when the two sets of scale parameters weakly majorize each other instead of usual majorization which by Lemma \ref{m2} is true under weak \textit{f} majorization of the scale parameters.
\begin{theorem}\label{st1}
{\rm Suppose $ X_i $ and $ X^{*}_i $ as in the setting of Theorem \ref{st1}. If $ xr(x) $ is increasing in $ x $, $ x^{2}r^{\prime}(x) $ is decreasing (increasing) in $ x $ and $ (\lambda_1,\dots,\lambda_n) \stackrel{\rm w}{\succeq}(\succeq_{\rm w})
(\lambda^{*}_{1}, \ldots , \lambda^{*}_{n}) $, then
\begin{itemize}
\item[(i)]  $X_{1:n}\geq_{\rm hr} (\leq_{\rm hr}) X^{*}_{1:n}$, and
\item[(ii)] if $ r(x) $ is decreasing then $ X_{1:n}\geq_{\rm disp} (\leq_{\rm disp}) X^{*}_{1:n} $.
\end{itemize}
}
\end{theorem}
\begin{proof}
(i) Fix $ x>0 $. Then the hazard rate of $ X_{1:n} $ is
\[
r_{X_{1:n}}(x,\boldsymbol{\lambda})=\sum_{i=1}^{n}\lambda_{i}r(\lambda_{i}x)=\frac{\sum_{i=1}^{n}\varphi(\lambda_{i}x)}{x},
\]
where $ \varphi(u)=ur(u) $, $ u\geq0 $. From Theorem 3.A.8 of \cite{met}, it suffices to show that, 
for each $ x>0 $, $ r_{X_{1:n}}(x, \boldsymbol{\lambda}) $  is Schur-concave (Schur-convex) and increasing in $ \lambda_{i}^{\prime} $s.
By the assumptions, $ \varphi(u) $ is increasing in $ u $, then the hazard
rate function of $ X_{1:n} $ is increasing in each $ \lambda_{i} $. 

Now, from Proposition 3.C.1 of \cite{met}, the concavity (convexity) of  $ \varphi(\lambda_{i}x) $ is
needed to prove Schur-concavity (Schur-convexity) of $ r_{X_{1:n}}(x,\boldsymbol{\lambda}) $. Note that the assumption that $ u^{2}r^{\prime}(u) $ is decreasing (increasing) in $ u $ is equivalent to $ r(u) + ur^{\prime}(u) $ is decreasing (increasing) in $ u $ since
\begin{equation*}
\left[ u^{2}r^{\prime}(u) \right]^{\prime} = u(2r^{\prime}(u)+ur^{\prime \prime}(u))
 = u\left[r(u) + ur^{\prime}(u)\right]^{\prime},
\end{equation*}
and $ r(u) + ur^{\prime}(u) $ is decreasing (increasing) in $ u $ is equivalent to $ ur(u) $ is concave (convex) in $ u $ since
\begin{equation*}
\left[ur(u)\right]^{\prime} = r(u) + ur^{\prime}(u).
\end{equation*}
Hence, $ \varphi(u) $ is concave (convex). This completes the proof of part (i).

(ii) Using the assumption that $ r(x) $ is decreasing in $ x $ and part (i), the required result follows from Theorem 2.1 in \cite{bag} and Theorem 1 in \cite{bar}.
\end{proof}

Note that the conditions of Theorem \ref{st1} are satisfied by the generalized gamma distribution as \cite{kh} proved that for $ X\sim GG(p,q) $, $ xr(x) $ an increasing function of $ x $ and $ x^{2}r^{\prime}(x) $ is an
increasing function of $ x $ when $ p,q> 1 $ and is a decreasing function of $ x $ when $ p,q<1 $. Recall that a random variable $ X $ has a generalized gamma distribution, denoted by $ X\sim GG(p,q) $, when its density function has the following form
\[
g(x)=\frac{p}{\Gamma(\frac{q}{p})}x^{q-1}e^{-x^{p}}, \quad x>0,
\]
where $ p,q>0 $ are the shapes parameters. The conditions of Theorem \ref{st1} are also satisfied by the Weibull distribution because for $ X_{i} \thicksim W(\alpha, \lambda) $, $ xr(x) $ is an increasing function of $ x $ and $ x^{2}r^{\prime}(x) $ is an
increasing function of $ x $ when $ \alpha \geq1 $ and is a decreasing function of $ x $ when $ \alpha\leq1 $, so Theorem \ref{st1} is also a generalization of Theorem 2.3 of \cite{k}.

Lastly, we get some new results on the lifetimes of parallel systems in terms of the usual stochastic order. It is noteworthy that
\cite{kh} in Theorem 2.1 proved Theorem \ref{st3} when $ f(x)=\log(x) $ and
\cite{kk} in Theorem 2.2 proved Theorem \ref{st3} when the baseline distribution in the scale model is exponential and $ f(x)=\log(x) $.

\begin{theorem}\label{st3}
{\rm Let $ X_{1}, \ldots , X_{n} $  be a set of independent nonnegative random variables with $ X_i\sim G(\lambda_i x)  $, $  i=1,\dots,n $. Let $ X^{*}_{1}, \ldots , X^{*}_{n} $ be another set of independent nonnegative
random variables with $ X^{*}_i\sim G(\lambda^{*}_i x) $, $ i=1,\dots,n $. If $ (f^{-1})^{\prime}(y)\tilde{r}(f^{-1}(y)) $ is decreasing in $ y $, where $ f $ is a strictly increasing function, then 
\begin{equation}\label{sq22}
 (\lambda_1,\dots,\lambda_n) \stackrel{\rm wf}{\succeq}
(\lambda^{*}_{1}, \ldots , \lambda^{*}_{n})\Longrightarrow X_{n:n}\geq_{\rm st} X^{*}_{n:n}.
\end{equation}
}
\end{theorem}
\begin{proof}
{\rm The survival function of $ X_{n:n} $ can be written as
\begin{equation}\label{sq2233}
\bar{G}_{X_{n:n}}(t, \boldsymbol{a})= 1- \prod_{i=1}^{n} G(f^{-1}(a_{i})t)
\end{equation}
where $ a_{i}=f(\lambda_{i}) $, for $  i = 1, \ldots , n $. Using Lemma \ref{mkl}, it is enough to show that the function $ \bar{G}_{X_{n:n}}(t, \boldsymbol{a}) $ given in \eqref{sq2233} is Schur-convex and decreasing in $ a_{i}^{\prime} $s. To prove its Schur-convexity, it follows from Theorem 3.A.4. in \cite{met} that we
have to show that for $ i\neq j $,
\[
(a_{i}-a_{j})\bigg(\frac{\partial \bar{G}_{X_{n:n}}}{\partial a_{i}}-\frac{\partial \bar{G}_{X_{n:n}}}{\partial a_{j}} \bigg)\geq 0,
\]
that is,for $ i\neq j $,
\begin{equation}\label{sq32}
(a_{i}-a_{j}) \prod_{k=1}^{n} G(f^{-1}(a_{k})t)\bigg(t(f^{-1})^{\prime}(a_{j}) \frac{g(f^{-1}(a_{j})t)}{G(f^{-1}(a_{j})t)}-t(f^{-1})^{\prime}(a_{i}) \frac{g(f^{-1}(a_{i})t)}{G(f^{-1}(a_{i})t)}\bigg)\geq 0.
\end{equation}
The assumption $ (f^{-1})^{\prime}(y)\tilde{r}(f^{-1}(y)) $ is decreasing in $ y $ implies that the function $ t(f^{-1})^{\prime}(a_{i})\tilde{r}(f^{-1}(a_{j})t) $ is decreasing in $ a_{i} $, for $  i = 1, \ldots , n $, from which it follows
that \eqref{sq32} holds. The partial derivative of $ \bar{G}_{X_{n:n}}(t, \boldsymbol{a}) $ with respect to $ a_{i} $ is negative,which in turn implies that the survival
function of $ X_{n:n} $ is decreasing in $ a_{i} $ for $ i = 1, \ldots , n $. This completes the proof of the required result.}
\end{proof}

\subsection{\rm Comparison of extreme values for exponentiated scale model}

Recall that random variable $X$ belongs to the exponentiated scale family of distributions if $ X\sim H(x) = [G(\lambda x)]^\alpha $, where $\alpha, \lambda>0 $ and $ G $ is called the baseline distribution and is an
absolutely continuous distribution function. We denote this family by $ {\rm ES}(\alpha, \lambda) $. Bashkar et al. \cite{btr} discussed stochastic comparisons of extreme order statistics from independent heterogeneous exponentiated scale samples.
In this section we provide new results for the comparison of smallest order statistics from samples following exponentiated scale model.
In the following theorem, we compare series systems with independent
heterogeneous ES components when one of the parameters is fixed, and the
results are then developed with respect to the other parameter. Again by Lemma \ref{m2}, this result are true under weak $ f $-supermajorization where $ f $ is a non-negative strictly increasing convex function.

\begin{theorem}\label{es11}
{\rm Let $ X_{1}, \ldots , X_{n} $ ($ X^{*}_{1}, \ldots , X^{*}_{n} $) be independent random variables with $ X_{i} \thicksim {\rm ES}(\alpha_i, \lambda)$ ($ X^{*}_{i} \thicksim {\rm ES}(\alpha^{*}_{i}, \lambda)$), $i = 1, . . . , n $. Then, for any $ \lambda>0 $, we
have
\[
\boldsymbol{\alpha}\stackrel{\rm w}{\succeq}\boldsymbol{\alpha^{*}}  \Longrightarrow 
 X_{1:n}\leq_{\rm hr}X^{*}_{1:n}.
\]
}
\end{theorem}
\begin{proof} 
Fix $ x>0 $. Then the hazard rate of $ X_{1:n} $ is
\[
r_{X_{1:n}}(x, \boldsymbol{\alpha}, \lambda)=\sum_{i=1}^{n}\alpha_{i}\lambda g(\lambda x) \dfrac{(G(\lambda x))^{\alpha_{i}-1}}{1-(G(\lambda x))^{\alpha_{i}}} =\lambda r(\lambda x) \sum_{i=1}^{n}\varphi(\alpha_{i},G(\lambda x)),
\]
where $ \varphi(x, p)= \dfrac{x p^{x}}{1-p^{x}} $, $ x\geq0, 0\leq p<1 $. From Theorem 3.A.8 of \cite{met}, it suffices to show that, 
for each $ x>0 $, $ r_{X_{1:n}}(x, \boldsymbol{\alpha}, \lambda) $  is Schur-convex and decreasing in $ \alpha_{i}^{\prime} $s.

By the Lemma 2.8 of \cite{to155}, $ \varphi(x, p) $ is decreasing and convex in $ x\geq0 $, then the hazard
rate function of $ X_{1:n} $ is decreasing and convex in each $ a_{i} $. 

So, from Proposition 3.C.1 of \cite{met}, the Schur-convexity of  $ r_{X_{1:n}}(x, \boldsymbol{\alpha}, \lambda) $
follows from convexity of $ \varphi(x, p) $.This completes the proof of the Required result.
\end{proof}

Recall that, a random variable $X$ is said to be distributed according the generalized exponential distribution, and will be denoted by $X\sim GE(\alpha, \lambda)$, if the distribution function is given by
\[
G(x)=\left(1-\exp\{-\lambda x\}\right)^{\alpha}, x>0, 
\]
where $\alpha >0$ is a shape parameter and $\lambda>0$ is a scale parameter. GE distribution is a member of ES family with underlying distribution $ G(x) = 1-\exp\{- x \}$. Therefore, we can get the following corollary that generalizes
the corresponding result in Theorem 15 of \cite{bee}. In particular the majorization assumption is relaxed to the weak supermajorization.
\begin{corollary}
{\rm For $ i=1,\dots,n $, let $ X_{i} $ and $ X^{*}_{i} $ be two sets of mutually independent random variables with $ X_i\sim {\rm GE}(\alpha_i , \lambda) $ and $ X^{*}_i\sim {\rm GE}(\alpha^{*}_i , \lambda) $. Then, for any $ \lambda>0 $, we
have
\[
 \boldsymbol{\alpha}\stackrel{\rm w}{\succeq}\boldsymbol{\alpha^{*}}  \Longrightarrow 
 X_{1:n}\leq_{\rm hr}X^{*}_{1:n}.
\]
}
\end{corollary}
 The following result considers the comparison on the lifetimes of series systems in terms
of the usual stochastic order when two sets of scale parameters weakly majorize each other.
\begin{theorem}
\rm Let $ X_{1}, \ldots , X_{n} $ ($ X^{*}_{1}, \ldots , X^{*}_{n} $) be independent random variables with $ X_{i} \thicksim {\rm ES}(\alpha, \lambda_{i} )$ ($ X^{*}_{i} \thicksim {\rm ES}(\alpha, \lambda^{*}_{i})$), $i = 1, . . . , n $. If $ q(\alpha, x) = \alpha \tilde{r}(x)\dfrac{G^{\alpha}(x)}{1-G^{\alpha}(x)}$ is decreasing (increasing) in $ x $, $\boldsymbol{\lambda} \stackrel{\rm w}{\succeq}(\succeq _{\rm w})
\boldsymbol{\lambda^{*}}$, then $X_{1:n}\geq_{\rm st}(\leq_{\rm st})X^{*}_{1:n}$.
\end{theorem}
\begin{proof}
For a fixed $x>0$, the survival function of $ X_{1:n} $ can be written as
\begin{equation}\label{gee1}
\overline{F}_{X_{1:n}}(x, \boldsymbol{\lambda})=  \prod_{i=1}^{n}\big(1-G(\lambda_{i}x))^{\alpha}\big).
\end{equation}

Now, using Theorem 3.A.8 of \cite{met}, it is enough to show that the function $ \overline{F}_{X_{1:n}}(x, \boldsymbol{\lambda}) $ given in \eqref{gee1} is Schur-convex (Schur-concave) and decreasing in $ \lambda_{i}^{\prime} $s. 

The partial derivatives of $ \overline{F}_{X_{1:n}}(x,\boldsymbol{\lambda}) $ with respect to $ \lambda_{i} $ is given by
\[
\frac{\partial \overline{F}_{X_{1:n}}(x, \boldsymbol{\lambda})}{\partial \lambda_{i}}=-x\overline{F}_{X_{1:n}}(x, \boldsymbol{\lambda})q(\alpha,  \lambda_{i}x),
\]
where $ q(\alpha, x) = \alpha \tilde{r}(x)\dfrac{G^{\alpha}(x)}{1-G^{\alpha}(x)}>0 $. Then we have that $ \overline{F}_{X_{1:n}}(x, \boldsymbol{\lambda}) $ is decreasing in each $ \lambda_{i} $.

From Theorem 3.A.4. in \cite{met} the Schur-convexity (Schur-concavity) follows if we prove that, for any $ i\neq j $,
\[
(\lambda_{i}-\lambda_{j})\bigg(\frac{\partial \overline{F}_{X_{1:n}}(x,\boldsymbol{\lambda})}{\partial \lambda_{i}}-\frac{\partial \overline{F}_{X_{1:n}}(x,\boldsymbol{\lambda})}{\partial \lambda_{j}} \bigg)\geq (\leq) 0,
\]
that is, for $ i\neq j $,
\begin{equation}
x\overline{F}_{X_{1:n}}(x, \boldsymbol{\lambda})(\lambda_{i}-\lambda_{j})\bigg(q(\alpha,  \lambda_{j}x)-q(\alpha,  \lambda_{i}x)\bigg)\geq(\leq) 0.
\end{equation}
By the assumption $ q(\alpha, x) $ is decreasing (increasing) in $ x $, which in turn implies that the
function $ q(\alpha,  \lambda_{i}x) $ is decreasing (increasing) in $ \lambda_{i} $ for $ i = 1, \ldots , n $. This completes the proof of the required result.
\end{proof}
  According to Lemma \ref{beel}, for the GE distribution $ q(\alpha, x) = h(\alpha, 1-\exp\{-x\}) $ is decreasing (increasing) in $ x $  for any $ 0<\alpha\leq1 $ ($ \alpha\geq1 $), so we have the following corollary.
\begin{corollary}\label{cf}
{\rm Let $ X_{1}, \ldots , X_{n} $ ($ X^{*}_{1}, \ldots , X^{*}_{n} $) be independent random variables with $ X_{i} \thicksim GE(\alpha, \lambda_{i} )$ ($ X^{*}_{i} \thicksim GE(\alpha, \lambda^{*}_{i})$), $i = 1, . . . , n $. If $ 0<\alpha\leq1 (\alpha \ge 1)$ and $(\lambda_1,\ldots,\lambda_n) \stackrel{\rm w}{\succeq}(\succeq_{\rm w})
(\lambda^{*}_{1}, \ldots , \lambda^{*}_{n})$, then $X_{1:n}\geq_{\rm st}(\leq_{\rm st})X^{*}_{1:n}$.}
\end{corollary}

Note that $ (\lambda_1,\ldots,\lambda_n) \stackrel{\rm m}{\succeq}
(\lambda^{*}_{1}, \ldots , \lambda^{*}_{n}) $ implies both $ (\lambda_1,\ldots,\lambda_n) \stackrel{\rm w}{\succeq}(\lambda^{*}_{1}, \ldots , \lambda^{*}_{n}) $\\ and $ (\lambda_1,\ldots,\lambda_n) \succeq_{\rm w}
(\lambda^{*}_{1}, \ldots , \lambda^{*}_{n}) $, Corollary \ref{cf} substantially improves the corresponding ones provided  by Balakrishnan et al. \cite{bee}, in the sense that the majorization is relaxed to the weak majorization. Naturally, one may wonder whether the following two statements
are actually also true: (i) For $ \alpha \ge 1 $, $ (\lambda_1,\ldots,\lambda_n) \underset{\exp}{\succeq _{\rm w}}
(\lambda^{*}_{1}, \ldots , \lambda^{*}_{n}) $ gives rise to the usual stochastic order between $ X_{1:n} $ and $ X^{*}_{1:n} $; (ii) For $ 0<\alpha\leq1 $, $ (\lambda_1,\ldots,\lambda_n) \stackrel{\rm p}{\succeq}
(\lambda^{*}_{1}, \ldots , \lambda^{*}_{n}) $ gives rise to the usual stochastic order between $ X_{1:n} $ and $ X^{*}_{1:n} $. The following example
gives negative answers to these two
conjectures.
\begin{example}{\rm
Let $ (X_{1}, X_{2}) $ ($ (X^{*}_{1}, X^{*}_{2}) $) be a vector of independent heterogeneous GE random variables.
\begin{itemize}
\item[(i)] Set $ \alpha = 2 $, $ (\lambda_{1}, \lambda_{2}) = (4, 0.5) $ and $ (\lambda^{*}_{1}, \lambda^{*}_{2}) = (2, 3) $. Obviously $ (\lambda_{1}, \lambda_{2})\underset{\exp}{\succeq _{\rm w}}(\lambda^{*}_{1}, \lambda^{*}_{2})  $, however $ X_{1:2} $ and $ X^{*}_{1:2} $ are not ordered in the usual stochastic order as
can be seen in Fig. \ref{fig:expm}.
\item[(ii)] Set $ \alpha = 0.6 $. For $ (\lambda_1, \lambda_2) = (1, 5.5)\stackrel{\rm p}{\succeq} (2,3) =(\lambda^{*}_{1},  \lambda^{*}_{2}) $, $ X_{1:2}\leq_{\rm st}X^{*}_{1:2} $; however, for $ (\lambda_1, \lambda_2) = (1, 2.25)\stackrel{\rm p}{\succeq} (1.1,2.14) =(\lambda^{*}_{1},  \lambda^{*}_{2}) $, $ X_{1:2}\geq_{\rm st}X^{*}_{1:2} $.  So, $ (\lambda_1, \lambda_2) \stackrel{\rm p}{\succeq} (\lambda^{*}_{1},  \lambda^{*}_{2}) $  implies neither $ X_{1:2}\leq_{\rm st}X^{*}_{1:2} $ nor $ X_{1:2}\geq_{\rm st}X^{*}_{1:2} $ for $ 0<\alpha\leq1 $.
\end{itemize}
}
\end{example}

\begin{figure}[ht]
\centering
\includegraphics[scale=.5]{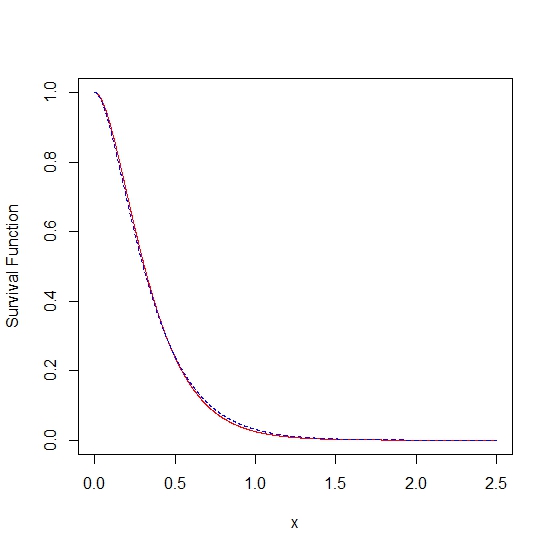}
%\captionsetup{textfont=rm,justification=centering,labelsep=newline}
\caption{\label{fig:expm} Plot of the survival
functions of $ X_{1:2} $ (dashed line) and $ X^{*}_{1:2} $ (continuous line) when $ \alpha = 2 $, $ (\lambda_{1}, \lambda_{2}) = (4, 0.5) $ and
$ (\lambda^{*}_{1}, \lambda^{*}_{2}) = (2, 3) $ for random
variables with GE distributions.}
\end{figure}

\subsection{Dependent samples with Archimedean structure}
Recently, some efforts are made to investigate stochastic comparisons on order statistics of random variables
with Archimedean copulas. See, for example, \cite{btr}, \cite{ll}, \cite{lfr} and
\cite{flld}. In this section we derive new result on the usual stochastic order between extreme order statistics of two heterogeneous random vectors with the dependent components having ES marginals and Archimedean copula structure.
Specifically, by $ \boldsymbol{X} \sim {\rm ES}(\boldsymbol{\alpha}, \boldsymbol{\lambda}, \phi ) $ we denote the sample having the Archimedean copula
with generator $ \phi $ and for $i=1,...,n$, $ X_i\sim ES(\alpha_i, \lambda_i ) $.

Here, we derive new result on the usual stochastic order for largest order statistics of samples ES and Archimedean sructure. The largest order statistic $ X_{n:n} $ of the sample $ \boldsymbol{X}\sim {\rm ES}(\boldsymbol{\alpha}, \lambda, \phi_{1} ) $ gets distribution function
\begin{equation}\label{four}
G_{X_{n:n} }(x)  = \phi \big( \sum_{i=1}^{n}\psi(G^{\alpha_{i}}(\lambda x))\big) = J(\boldsymbol{\alpha}, \lambda, x, \phi)
\end{equation}

 \begin{theorem}\label{de1}
{\rm For $ \boldsymbol{X} \sim {\rm ES}(\boldsymbol{\alpha}, \lambda, \phi_{1} ) $ and $ \boldsymbol{X^{*}} \sim {\rm ES}(\boldsymbol{\alpha^{*}}, \lambda, \phi_{2} ) $,
\begin{itemize}
\item[(i)] if $ \phi_{1} $ or $ \phi_{2} $ is log-convex, and $ \psi_2\circ\phi_1 $ is super-additive, then $ \boldsymbol{\alpha} \succeq _{\rm w}
\boldsymbol{\alpha^{*}} $ implies $X_{n:n}\geq_{\rm st}X^{*}_{n:n}$;
\item[(ii)] if $ \phi_{1} $ or $ \phi_{2} $ is log-concave, and $ \psi_1\circ\phi_2 $ is super-additive, then $ \boldsymbol{\alpha} \stackrel{\rm w}{\succeq}
\boldsymbol{\alpha^{*}} $ implies $X_{n:n}\leq_{\rm st}X^{*}_{n:n}$.
\end{itemize}}
\end{theorem}
\begin{proof}
According to Equation \eqref{four}, $ X_{n:n} $ and $ X^{*}_{n:n} $ have their respective distributin functions $ J(\boldsymbol{\alpha}, \lambda, x, \phi_1) $ and $ J(\boldsymbol{\alpha^{*}}, \lambda, x, \phi_2) $,
for $ x \geq 0 $.

\begin{itemize}
\item[(i)]We only prove the case that $ \phi_1 $ is log-convex, and the other case can be finished similarly. First we show that $ J(\boldsymbol{\alpha}, \lambda, x, \phi_1) $ is decreasing and Schur-concave function of $ \alpha_i, i=1,\ldots,n$. Since $ \phi_1 $ is decreasing, we have
\[
\dfrac{\partial J(\boldsymbol{\alpha}, \lambda, x, \phi_1)}{\partial \alpha_i} = \log(G(\lambda x)) (G(\lambda x))^{\alpha_i} \dfrac{\phi^{\prime}_1 \big( \sum_{i=1}^{n}\psi_{1}(G^{\alpha_i}(\lambda x))\big)}{\phi^{\prime}_1 \big( \psi_1(G^{\alpha_i}(\lambda x))\big)} \leq 0,
\]
\[
 \text{for all}\quad x>0,
\]	
That is, $ J(\boldsymbol{\alpha}, \lambda, x, \phi_1) $ is decreasing in $ \alpha_{i} $ for $ i = 1, \ldots , n $.
Furthermore, for $ i\neq j $, the decreasing $ \phi_1 $ implies
\[
(\alpha_i - \alpha_j) \big(\dfrac{\partial J(\boldsymbol{\alpha}, \lambda, x, \phi_1)}{\partial \alpha_i}-\dfrac{\partial J(\boldsymbol{\alpha}, \lambda, x, \phi_1)}{\partial \alpha_j}\big) = 
\]
\[
(\alpha_i - \alpha_j)\log(G(\lambda x)) \phi^{\prime}_1 \big( \sum_{i=1}^{n}\psi_{1}(G^{\alpha_i}(\lambda x))\big)
\]
\begin{equation*}
 \bigg(
 \dfrac{(G(\lambda x))^{\alpha_i}}{\phi^{\prime}_1 \big( \psi_1(G^{\alpha_i}(\lambda x))\big)}- \dfrac{(G(\lambda x))^{\alpha_j}}{\phi^{\prime}_1 \big( \psi_1(G^{\alpha_j}(\lambda x))\big)} \bigg)
\end{equation*}
\begin{equation*}
\mathop  = \limits^{{\mathop{\rm sgn}} } (\alpha_i - \alpha_j)   \bigg(
 \dfrac{(G(\lambda x))^{\alpha_i}}{\phi^{\prime}_1 \big( \psi_1(G^{\alpha_i}(\lambda x))\big)}- \dfrac{(G(\lambda x))^{\alpha_j}}{\phi^{\prime}_1 \big( \psi_1(G^{\alpha_j}(\lambda x))\big)} \bigg).
\end{equation*}
where $ \mathop = \limits^{{\mathop{\rm sgn}} } $ means that both sides have the sam sign. 
Note that the log-convexity of $ \phi_1 $ implies the decreasing property of $ \frac{\phi_1}{\phi^{\prime}_1} $. Since
$ \psi_1(G^{\alpha_i}(x)) $ is increasing in $ \alpha_{i} $, then $ \dfrac{G^{\alpha_i}(x)}{\phi^{\prime}_1 \big( \psi_1((G(x))^\alpha_i)\big)} = \dfrac{\phi_1(\psi_1((G(x))^{\alpha_i}))}{\phi^{\prime}_1 \big( \psi_1((G(x))^{\alpha_i})\big)} $ is decreasing in $ a_{i} $.
So, for $ i\neq j $,
\[
(\alpha_i - \alpha_j) \big(\dfrac{\partial J(\boldsymbol{\alpha}, \lambda, x, \phi_1)}{\partial \alpha_i}-\dfrac{\partial J(\boldsymbol{\alpha}, \lambda, x, \phi_1)}{\partial \alpha_j}\big)\leq 0.
\]
Then Schur-concavity of $ J(\boldsymbol{\alpha}, \lambda, x, \phi_1) $ follows from Theorem 3.A.4. in \cite{met}. According to Theorem 3.A.8 of \cite{met}, $ \boldsymbol{\alpha} \succeq _{\rm w}
\boldsymbol{\alpha^{*}} $
implies $ J(\boldsymbol{\alpha}, \lambda, x, \phi_1) \leq J(\boldsymbol{\alpha^{*}}, \lambda, x, \phi_1)  $.
On the other hand, since $ \psi_2\circ\phi_1 $ is super-additive
by Lemma A.1. of \cite{lfr}, we have $ J(\boldsymbol{\alpha^{*}}, \lambda, x, \phi_1) \leq J(\boldsymbol{\alpha^{*}}, \lambda, x, \phi_2)  $. So, it holds that
\[
J(\boldsymbol{\alpha}, \lambda, x, \phi_1) \leq J(\boldsymbol{\alpha^{*}}, \lambda, x, \phi_1) \leq J(\boldsymbol{\alpha^{*}}, \lambda, x, \phi_2) .
\]
That is, $ X_{n:n}\geq_{\rm st} X^{*}_{n:n} $.
\item[(ii)] We omit its proof due to the similarity to that of Part (i).
\end{itemize}
\end{proof}

Note that Theorem \ref{de1} for particular case $ \lambda = 1 $ in \cite{flld} has been proved.

From Theorem \ref{de1} (i) and the fact that weak log-majorization implies weak submajorization, we readily obtain the following corollary.
\begin{corollary}\label{corr}
{\rm For $ \boldsymbol{X} \sim {\rm ES}(\boldsymbol{\alpha}, \lambda, \phi_{1} ) $ and $ \boldsymbol{X^{*}} \sim {\rm ES}(\boldsymbol{\alpha^{*}}, \lambda, \phi_{2} )$,

if $ \phi_{1} $ or $ \phi_{2} $ is log-convex, and $ \psi_2\circ\phi_1 $ is super-additive, then $ \boldsymbol{\alpha} \underset{\log}{\succ _{\rm w}}
\boldsymbol{\alpha^{*}} $  implies $X_{n:n}\leq_{\rm st}X^{*}_{n:n}$.
}
\end{corollary}
Letting $ \lambda = 1 $ in Corollary \ref{corr} leads to the following corollary for PRH samples.
\begin{corollary}\label{c1}
{\rm For $ \boldsymbol{X} \sim {\rm PRH}(\boldsymbol{\alpha}, \phi_{1} ) $ and $ \boldsymbol{X^{*}} \sim {\rm PRH}(\boldsymbol{\alpha^{*}}, \lambda, \phi_{2} )$,

if $ \phi_{1} $ or $ \phi_{2} $ is log-convex, and $ \psi_2\circ\phi_1 $ is super-additive, then $ \boldsymbol{\alpha} \underset{\log}{\succ _{\rm w}}
\boldsymbol{\alpha^{*}} $  implies $X_{n:n}\leq_{\rm st}X^{*}_{n:n}$.}
\end{corollary}
Note that \cite{flld} in Theorem 5.2 proved the stochastic order between two largest order statistics when $ -\log \phi_{1} $ or $ -\log \phi_{2} $ is log-concave, but according to Corollary \ref{c1}, we do not need to check the log-concavity of $ -\log \phi_{1} $ or $ -\log \phi_{2} $ and it is only enough that $ \phi_{1} $ or $ \phi_{2} $ be log-convex.

\section{Conclusions}
We used a new majorization notion, called $f$-majorization. The new majorization notion includes, as special cases, the usual majorization, the reciprocal majorization and the $ p $-larger majorization notions. We provided a comprehensive account of the mathematical properties of the $f$-majorization order and gave applications of this order in the context of stochastic comparison of extreme order statistics.

\section*{Acknowledgement}
We would like to express our deep appreciation to Dr Javanshiri for his  helpful comments that
improved this paper.  

\end{document}